\documentclass[reqno, 10pt]{amsart}
\usepackage{bbm}
\usepackage{amsfonts}
\usepackage{mathrsfs}
\usepackage{amscd}
\usepackage{cases}
\usepackage{mathrsfs}
 \usepackage[dvipsnames]{xcolor}
\usepackage[top=4cm, bottom=3cm, left=3.2cm, right=3.2cm]{geometry}
\newtheorem{thm}{Theorem}

\newtheorem{lem}{Lemma}

\theoremstyle{definition}

\theoremstyle{remark}
\newtheorem{rem}{Remark}

\numberwithin{equation}{section} \numberwithin{lem}{section}
\numberwithin{thm}{section} \numberwithin{prop}{section}
\numberwithin{cor}{section} \numberwithin{rem}{section}\numberwithin{hyp}{section}

\title[The Euler-alignment system]{
	On the global classical solution to compressible Euler system
        with singular  velocity alignment}

\author{Li Chen}
\address{Li Chen, Department of Mathematics, University of Mannheim,
  68131 Mannheim, Germany. }
\email{chen@math.uni-mannheim.de}

\author{Changhui Tan}
\address{Changhui Tan, Department of Mathematics, University of South
  Carolina, Columbia, SC 29208, USA.}
\email{tan@math.sc.edu}

\author{Lining Tong$^*$}\thanks{*Corresponding author}
\address{Lining Tong, Department of Mathematics, Shanghai University, 200244, China.}
\email{tongln@shu.edu.cn}

\date{\today}

\thanks{\emph{Acknowledgment.}
  The research of CT is supported by the NSF grant DMS 1853001.
  The research of LT is supported by the NSFC grants No. 11771274, 11901379.}

\begin{document}
\maketitle

\centerline{\it Dedicated to Professor Ling Hsiao's 80th birthday}

\begin{abstract}
 We consider a compressible Euler system with singular velocity
 alignment, known as the \emph{Euler-alignment system}, describing the
 flocking behaviors of large animal groups.
 We establish a local well-posedness theory for the system, as well as
 a global well-posedness theory for small initial data.
 We also show the asymptotic flocking behavior, 
 where solutions converge to a constant steady state exponentially in time.
\end{abstract}

\noindent{\tiny AMS SUBJECT CLASSIFICATIONS:} 35Q35, 35Q70, 35L65.

\medskip

\section{Introduction}

In this paper, we study the following Cauchy problem
\begin{eqnarray}
&&\label{eqn1o}\partial_t \rho+\nabla\cdot(\rho u)=0,\\
&&\label{eqn2o}\partial_t (\rho u)+\nabla\cdot(\rho u\otimes u)+\nabla p(\rho)=-\beta\rho u-\rho\int_\Omega\phi(x-y)(u(x)-u(y))\rho(y)\mathrm{d}y,
\end{eqnarray}
in $(0,\infty)\times\Omega$ with initial conditions
\begin{eqnarray}\label{eqn3}
\rho|_{t=0}=\rho_0(x),\quad u|_{t=0}=u_0(x),\quad x\in \Omega.
\end{eqnarray}
The spatial domain $\Omega$ can be either the whole space
$\mathbb{R}^N$ or the torus $\mathbb{T}^N$, where $N$ denotes the dimension.
$\rho$ and $u$ are the unknown density and velocity, respectively.
The pressure is given by the power law $p(\rho)=\rho^{\gamma}$ with
$\gamma\geq 1$, and the damping coefficient $\beta \geq 0$.
The last term in \eqref{eqn2o} represents the nonlocal velocity alignment,
where $\phi$ is called the \emph{communication weight}, measuring the
strength of the alignment interactions.

System \eqref{eqn1o}-\eqref{eqn2o} can be formally derived from a mean
field $M$-particle Newtonian interaction system of the type
\begin{equation}\label{micro}\begin{cases}
   ~\dot{X}_i(t)=V_i(t), \quad 1\leq i\leq M,\\
   ~\dot{V}_i(t)=\dfrac{1}{M}\displaystyle\sum_{j\neq i}F\Big(t, X_i(t)-X_j(t),V_i(t)-V_j(t)\Big)-\beta V_i(t),
 \end{cases}
\end{equation}
with the interacting force $F(t,x,v)=-\phi(x)v$.
It is known as the Cucker-Smale model \cite{CS} which describes
the flocking phenomenon for animal groups.
Other celebrated models that lie in the framework of \eqref{micro}
range from
the classical mechanics with Coulomb force in 3D \cite{HJ1}
$F(t,x,v)=x/|x|^3$,
to modeling the social behavior of agents, for example, wealth distribution in \cite{DLR2} and pedestrian flow in \cite{EGKT}.

Taking the mean field limit of the particle system \eqref{micro}, one obtains, in the mesoscopic level, the Vlasov type kinetic equation
\[
  \partial_t f+v\cdot \nabla_x f+\nabla_v\Big(\iint_{\Omega\times\Omega} F(t,x-y,v-v' )f(t,y,v' ) \mathrm{d}y \mathrm{d}v' f(t,x,v)\Big)-\beta\nabla_v \cdot(vf)=0.
\]
The rigorous derivation of the mean field model with different
backgrounds has been studied extensively in the last decades,
for example in \cite{BP, CCR, CCP, CFTV, CGY, HL, HJ1, LP}, to name a few.

The compressible Euler systems like \eqref{eqn1}-\eqref{eqn2} serve as
hydrodynamic limits of the kinetic equations.
Different choices of Ansatz lead to different pressure laws.
For instance,
mono-kinetic Ansatz  $f(t,x,v)=\rho(t,x)\delta_{u(t,x)}(v)$ implies
the pressureless system $p(\rho)\equiv0$;
local Maxwellian $f(t,x,v)=\rho(t,x)\frac{1}{(2\pi)^{\frac{N}{2}}}
e^{-|v-u(t,x)|^2/2}$ leads to linear pressure $p(\rho)=\rho$; and
the Ansatz $f(t,x,v)=\chi_{|v-u(t,x)|^2\leq \rho^{2/N} (t,x)}
(x,v)$, which comes from the minimization of kinetic energy $\iint
\frac{v^2}{2}f(t,x,v)\mathrm{d}x\mathrm{d}v$ under the restriction
$\|f(t,x,v)\|_{L^\infty}\leq 1$,  would yield the nonlinear pressure
$p(\rho)=\rho^{\frac{N+2}{N}}$. Rigorous justifications of
hydrodynamic limits on Cucker-Smale model can be found in
\cite{FK, KMT}.

It is well-known that for the compressible Euler system
\eqref{eqn1}-\eqref{eqn2} with $\beta=\phi=0$,
only local existence of smooth solutions could be expected,
and shock waves will form in finite time. Velocity damping helps in
preventing the formation of shocks (see \cite{STW, TW, WY}).

When the communication weight $\phi$ is bounded and has a positive
lower bound (or decays sufficiently slow at infinity in the whole
space case), the nonlocal velocity alignment has a damping effect,
which could restrain shock formation, for a class of subcritical
initial data.
See \cite{CCTT,TT14,T20} for threshold conditions for the pressureless
systems, and \cite{C} for isothermal pressure $p(\rho)=\rho$ with
small initial data.
For the case where the communication weight is not positive, such as in the pedestrian and material flow case, an additional damping effect is required to prove the global smooth solution for small initial data in \cite{CCGW,TC,TCGW}.

We are interested in the case when the communication weight $\phi$ is
singular at the origin. A prototype choice of $\phi$ would be
\begin{equation}\label{phi}
  \phi(x)=\frac{c_\alpha}{|x|^{N+2\alpha}},\quad
   c_\alpha=\frac{2^{2\alpha}\Gamma(\alpha+\frac{N}{2})}{\pi^{\frac{N}{2}}|\Gamma(-\alpha)|}.
\end{equation}
for $\alpha\in(0,1)$, where the constant $c_\alpha$ is related to the fractional Laplacian
operator, which in $\mathbb{R}^N$ reads
\begin{equation}\label{LD}
  \mathcal{L}=(-\Delta)^{\frac{1}{2}},\quad
  \mathcal{L}^{2\alpha}f=c_{\alpha}P.V.
  \int_{{\mathbb R}^N}\frac{f(x)-f(y)}{|x-y|^{N+2\alpha}}\mathrm{d}y,
  \quad 0<\alpha<1.
\end{equation}
Note that \eqref{LD} holds in $\Omega=\mathbb{T}^N$
by viewing $f$ as a periodic function in $\mathbb{R}^N$.

By considering a purturbation of the constant solution $\rho\equiv1$,
$u\equiv0$, the system \eqref{eqn1o}-\eqref{eqn2o} with $\phi$ defined in \eqref{phi}
can be reformulated into
\begin{eqnarray}
 &&\label{eqn1}\partial_t \rho+\nabla\cdot(\rho u)=0,\\
 &&\label{eqn2}\partial_t (\rho u)+\nabla\cdot(\rho u\otimes u)+\nabla p( \rho)=-\beta\rho u-\rho \mathcal{L}^{2\alpha}u-\rho \mathcal{L}^{2\alpha}((\rho-1) u)+\rho u\mathcal{L}^{2\alpha}\rho.
\end{eqnarray}
On can observe a linear fractional viscosity term
$-{\mathcal L}^{2\alpha}u$ in \eqref{eqn2}, which has
a regularization effect.
Such effect has been captured beautifully in 1D, where global
regularity can be shown for all smooth initial data away from vacuum
(see \cite{T19} for the effect of the vacuum),
for the pressureless system with $\alpha\in(0,1)$ \cite{DKRT,KT,ST1, ST3},
and for the isentropic system ($p(\rho)=\rho^\gamma $) with
$\alpha\in(\frac56,1)$ \cite{CDS}.
The multi-dimensional system, however, is much less understood,
due to the lack of an auxiliary quantity, first introduced in
\cite{CCTT}, that nicely captures the commutator structure (or
cancelation property) in
$-\mathcal{L}^{2\alpha}((\rho-1) u)+u\mathcal{L}^{2\alpha}\rho$,
so that it is dominated by the linear dissipation.
To our best knowledge, the only global result in multi-dimension is for the
pressureless system with small initial data \cite{Shv}.

The main goal of this paper is to establish a global theory for the
system \eqref{eqn1o}-\eqref{eqn2o} in multi-dimension with pressure.
First, we establish a local well-posedness theory, together with a
regularity criterion. The singular kernel leads to a regularization
effect for the velocity $u$.
Next, we prove global regularity for small initial data.
The main subtlety is to make use of the pressure to generate
dissipation for the density (the idea was introduced in \cite{STW}),
and to control the last two terms in \eqref{eqn2} together with
the velocity dissipation.
Finally, we show an exponential convergence of the solution towards
the constant steady state $\rho\equiv1$ and $u\equiv0$, when
$\Omega=\mathbb{T}^N$.
Unlike the pressureless dynamics, where the asymptotic density profile
is not necessarily uniformly distributed (e.g. \cite{Shv}), the presence of pressure enforces the
steady state to be a constant $\rho\equiv1$.
The exponential decay is obtained by a careful examination on the physical energy.

The arrangement of this paper is the following.
In section 2, we state the main results, together with several
preliminary lemmas.
In section 3, we establish the local well-posedness theory and
regularity criterion, using energy method.
In section 4, we show global well-posedness for small initial data. Finally, in section 5, we prove the exponential decay of the solution when the domain is a torus.

Here in the following we introduce several notations used throughout
the paper.
We will repeatedly use $C$ as a generic positive constant. Unless
specified, $C$ can depend on parameters $\gamma, \beta, s, N$, etc,
but is independent of $t$ and data $(\rho,u)$.
Denote $C^\lambda:=C^\lambda(\Omega)$ be the H\"older space, for
$\lambda\in(0,1)$, and $C^{\lambda}=C^{1,\lambda-1}$ for
$\lambda\in(1,2)$. For $\lambda=1$, we use $C^1$ to represent
$C^{1,\epsilon}$ for simple notations. See Remark \ref{rem:C1}
for more details.
For simplicity, we write $\int f\mathrm{d}x:=\int_{\Omega} f\mathrm{d}x$.

\section{Preliminaries and main results}
\subsection{Reformulation of the problem}\label{2.1}
We start by reformulating the Cauchy problem of the compressible Euler system
(\ref{eqn1o})-(\ref{eqn3}) with respect to the constant solution
$\rho\equiv1$ and $u\equiv0$, following the idea in \cite{STW}.

Introduce a new variable $\sigma$,  defined as follows
\begin{eqnarray}\label{sigma}
\sigma=\sigma(\rho):=\left\{\begin{array}{ll}\ln \rho &\gamma=1,\\ \frac{2\sqrt{\gamma}}{\gamma-1}(\rho^{\frac{\gamma-1}{2}}-1)&\gamma>1.
\end{array}\right.
\end{eqnarray}
Inversely, $\rho$ can be expressed by
\begin{eqnarray}\label{rho}
  \rho=\rho(\sigma):=\left\{\begin{array}{ll} e^\sigma &\gamma=1\\ \Big(\frac{\gamma-1}{2\sqrt{\gamma}}\sigma+1\Big)^{\frac{2}{\gamma-1}}&\gamma>1
\end{array}\right.
\end{eqnarray}
The equation \eqref{eqn1}-\eqref{eqn2} are transformed into the
following equivalent system:
\begin{align}
  &\label{eqn1a}\partial_t \sigma+u\cdot\nabla\sigma +
    \left(\frac{\gamma-1}{2}\sigma+\sqrt{\gamma}\right)\nabla\cdot u=0,\\
  &\label{eqn2a}\partial_t u+u\cdot\nabla u+
     \left(\frac{\gamma-1}{2}\sigma+\sqrt{\gamma}\right)\nabla \sigma=-\beta u-\mathcal{L}^{2\alpha}u-\mathcal{L}^{2\alpha}((\rho(\sigma)-1)u)+u\mathcal{L}^{2\alpha}(\rho(\sigma)-1).
\end{align}
subject to the initial data
\begin{equation}\label{eqn3a}
\sigma|_{t=0}=\sigma_0(x)=\sigma(\rho_0(x))\quad u|_{t=0}=u_0(x),\quad x\in \Omega,
\end{equation}

One can check that the $C^1$ solution of \eqref{eqn1a}-\eqref{eqn3a}
is equivalent to the solution of \eqref{eqn1o}-\eqref{eqn3}, if the
density is positive, namely
\[\rho_{\min}(t):=\inf_{x\in\Omega}\rho(t,x)=\inf_{x\in\Omega}\rho(\sigma(t,x))>0.\]

\subsection{Main results}
We study local and global well-posedness of the system  \eqref{eqn1a}-\eqref{eqn3a}.

The first result concerns the local well-posedness of the system.

\begin{thm}[Local well-posedness] \label{local}
  Let $s>\frac{N}{2}+\max\{1,2\alpha\}$,
  assume that $(\sigma_0(x),u_0(x))\in (H^{s}(\Omega))^{N+1}$ and $\inf_{x\in\Omega}\rho(\sigma_0(x))>0$.
  Then, there exist a unique classical solution $(\sigma,u)$ of the
  Cauchy problem \eqref{eqn1a}-\eqref{eqn3a} satisfying
  \begin{equation}\label{localreg}
   \sigma\in C([0,T],H^s(\Omega)), \quad u\in C([0,T],
   (H^s(\Omega))^N)\cap L^2([0,T], (H^{s+\alpha}(\Omega))^N)
  \end{equation}
  for some finite $T>0$. Moreover, \eqref{localreg} holds for any time
  $T$ if and only if
  \begin{equation}\label{BKM}
    \int_0^T\Big(\|\nabla
    u(t,\cdot)\|_{L^\infty}+\|\sigma(t,\cdot)\|_{C^{\max\{1,2\alpha\}}}\Big)\mathrm{d}t<+\infty.
  \end{equation}
\end{thm}
\begin{rem}
  Theorem \ref{local} holds for any smooth initial data, as long as
  density stays away from vacuum, in which case the linear dissipation
  $-\mathcal{L}^{2\alpha}u $ plays a dominate role in the alignment
  force. Note that the last two terms in \eqref{eqn2a} can be viewed
  as a commutator
  \begin{equation}\label{com}
    \mathcal{L}^{2\alpha}((\rho(\sigma)-1)u)-u\mathcal{L}^{2\alpha}(\rho(\sigma)-1)
    =[\mathcal{L}^{2\alpha}, u](\rho(\sigma)-1).
  \end{equation}
  One needs to make good use of this commutator structure in order to
  obtain the desired dissipative estimates. With the singular
  alignment force, the solution $u$ gains regularity instantly. This
  is a major difference compared with the regular (bounded) alignment force.
\end{rem}

Next, we turn to the global well-posedness theory.
One standard approach is to show that the Beale-Kato-Majda type regularity
criterion \eqref{BKM} holds in all finite time.
However, it is generally difficult to validate such criterion (except
in 1D with the aid of an additional structure \cite{CDS}).

We focus on global regularity for small initial data. We will show
that the smallness propagates in time, and hence condition \eqref{BKM}
holds in all finite time.

\begin{thm}[Global well-posedness for small data]\label{global}
  Let $s>\max\left\{\frac{N}{2}+\max\{1,2\alpha\}, \,2-\alpha\right\}$.
  Assume either $\Omega=\mathbb{R}^N$ and $\beta>0$, or
  $\Omega=\mathbb{T}^N$ and $\beta\geq0$.
  There exists a small parameter $\delta_0>0$, such that if
  \begin{equation}\label{smallinit}
    \|\sigma_0\|_{H^s}^2+\|u_0\|_{H^s}^2\leq \delta_0^2,
  \end{equation}
  then the Cauchy problem \eqref{eqn1a}-\eqref{eqn3a} has a unique global
  classical solution $(\sigma, u)$.
\end{thm}
\begin{rem}
  The damping is needed when $\Omega=\mathbb{R}^N$ in order to provide
  enough control on $\|u\|_{L^2}$. It is not required when
  $\Omega=\mathbb{T}^N$ as $\|u\|_{L^2}$ can be controlled by
  dissipation via Poincar\'e inequality.
  See Remark \ref{rem:longtime} for discussions on asymptotic flocking
  behaviors and convergence rate.
\end{rem}

Our final result is on the asymptotic behavior of the system. We show
the flocking phenomenon with fast alignment in the case
$\Omega=\mathbb{T}^N$. 
\begin{thm}[Large-time behavior]\label{large}
    Assume $\Omega=\mathbb{T}^N$, $\beta\geq0$, and $(\rho, u)$ be the
    classical solution to the system \eqref{eqn1o}-\eqref{eqn3} 
    satisfying $(\rho,u)\in L^{\infty}((0,+\infty)\times\mathbb{T}^N)$.
    We further assume the initial data $(\rho_0,u_0)$ satisfies
    \begin{equation}\label{rhobar}
      \frac{1}{|\mathbb{T}^N|}\int_{\mathbb{T^N}}\rho_0(x)\mathrm{d}x=1,
    \end{equation}
    and for the case $\beta=0$, we assume additionally
    \begin{equation}\label{zeromom}
      \int_{\mathbb{T}^N}\rho_0(x) u_0(x)\mathrm{d}x=0.
    \end{equation}
    Then there exists constants $\mu=\mu(\alpha,\|\rho\|_{L^\infty},
    \|u\|_{L^\infty})$ and $C=C(\alpha, \rho_0, u_0, \|\rho\|_{L^\infty}, \|u\|_{L^\infty})$
    such that
    \begin{equation}\label{expDecay}
      \int\rho(t,x)|u(t,x)|^2\mathrm{d}x+\int(\rho(t,x)
      -1)^2\mathrm{d}x\leq Ce^{-\mu t}.
    \end{equation}
    Moreover, if \eqref{smallinit} is satisfied for the corresponding
    $(\sigma, u)$ system \eqref{eqn1a}-\eqref{eqn3a}, then there
    exists constants $\mu=\mu(\alpha,\delta_0)$ and $C=C(\alpha,
    \delta_0)$ such that
    \begin{equation}\label{expDecayh}
      \|\sigma(t,\cdot)\|_{H^s}^2+\|u(t,\cdot)\|_{H^s}^2\leq Ce^{-\mu t}.
    \end{equation}
\end{thm}
\begin{rem}
  Conditions \eqref{rhobar} and \eqref{zeromom} is naturally required
  to obtain \eqref{expDecay},
  due to the conservation of mass, and the conservation of momentum
  (when $\beta=0$), respectively. These two conditions can be easily
  removed by scaling on $\rho$ and shifting on $u$, and the estimate \eqref{expDecay} will be
  replaced by
  \[\int\rho(t,x)|u(t,x)-\bar{u}|^2\mathrm{d}x+\int(\rho(t,x)
      -\bar{\rho})^2\mathrm{d}x\leq Ce^{-\mu t},\]
  where $\bar{\rho}$ is the average density, and $\bar{u}$ is the
  average velocity.
\end{rem}

\subsection{Elementary estimates}
Next, we state the fractional Leibniz rule and commutator estimates
that will be used. We refer to e.g. \cite[Theorem 1.2]{DL},
\cite[Lemma 6.1]{MTX} for more details.
\begin{lem}[Fractional Leibniz rule]\label{lem:Leibniz}
  For $\lambda>0$,
    there exists a constant $C=C(\lambda, N)$ such that
    \begin{equation}\label{Leibnitz}
      \|\mathcal{L}^{\lambda}(fg)\|_{L^2}\leq C(\|\mathcal{L}^\lambda f\|_{L^2}\|g\|_{L^\infty}+\|\mathcal{L}^\lambda g\|_{L^2}\|f\|_{L^\infty}),
    \end{equation}
\end{lem}

\begin{lem}[Commutator estimates]\label{lem:comm}
    For $\lambda>1$, there exists a constant $C=C(\lambda, N)$ such that
    \begin{align}
      \|[\mathcal{L}^{\lambda},f]g\|_{L^2}\leq&\,  C(\|\mathcal{L}^\lambda f\|_{L^2}\|g\|_{L^\infty}+\|\nabla f\|_{L^\infty}\|\mathcal{L}^{\lambda-1}g\|_{L^2}), \label{comm21}\\
      \|[\mathcal{L}^\lambda,f,g]\|_{L^2}\leq&\,
      C(\|\mathcal{L}^{\lambda-1} f\|_{L^2}\|\nabla g\|_{L^\infty}
      + \|\nabla f\|_{L^\infty} \|\mathcal{L}^{\lambda-1} g\|_{L^2}), \label{comm3}
    \end{align}
    where
    \[
      [\mathcal{L}^{\lambda},f]g=\mathcal{L}^{\lambda}(fg)-f\mathcal{L}^{\lambda}g,\quad
      [\mathcal{L}^{\lambda},f,g]=\mathcal{L}^{\lambda}(fg)-f\mathcal{L}^{\lambda}g-g\mathcal{L}^{\lambda}f.
    \]
    For $\lambda\in(0,1]$,  there exists a constant $C=C(\lambda, N)$ such that
  \begin{equation}\label{comm22}
	\|[\mathcal{L}^{\lambda},f]g\|_{L^2}\leq C\|\mathcal{L}^{\lambda}f\|_{L^\infty}\|g\|_{L^2},
  \end{equation}
\end{lem}

The following composition estimate is useful for handling the
nonlinear mappings between $\sigma$ and $\rho$. It indicates that
$\sigma$ and $\rho$ have the same regularity if $\rho$ is away from zero.
\begin{lem}[Composition estimates] Let $\lambda>0$. There exists a constant
  $C=C(\gamma,\lambda,\|\sigma\|_{L^\infty},\rho_{\min}^{-1})$ such that
  \begin{equation}\label{composition}
    \|\mathcal{L}^\lambda(\rho(\sigma)-1)\|_{L^2}\leq C(\gamma,\lambda,\|\sigma\|_{L^\infty},\rho_{\min}^{-1}) \|\mathcal{L}^\lambda\sigma\|_{L^2}.
  \end{equation}
\end{lem}
\begin{proof}
  We make use of the following composition estimate, the proof of
  which can be found, for instance, in \cite[Theorem A.1.]{LT}.
  \[\|\mathcal{L}^\lambda(\rho(\sigma))\|_{L^2}\leq
    C\|\rho\|_{C^{\lceil \lambda\rceil}(\text{supp}(\sigma))}(1+\|\sigma\|_{L^\infty}^\lambda)\|\mathcal{L}^\lambda\sigma\|_{L^2}.\]
  The term $\|\rho\|_{C^{\lceil \lambda\rceil}(\text{supp}(\sigma))}$ can be
  estimated via definition \eqref{rho}.
  For $\gamma=1$, $\frac{d^{\lceil \lambda\rceil}}{d\sigma^{\lceil
      \lambda\rceil}}\rho=e^\sigma$. Then,
  \[\|\rho\|_{C^{\lceil
        \lambda\rceil}(\text{supp}(\sigma))}\leq
      \exp\left(\|\sigma\|_{L^\infty}\right).\]
  For $\gamma>1$, we have
  \[\left|\frac{d^{\lceil \lambda\rceil}}{d\sigma^{\lceil
          \lambda\rceil}}\rho\right|\leq C(\gamma,
    \lambda)(\rho(\sigma))^{1-{\lceil \lambda\rceil}\frac{\gamma-1}{2}}
    \leq\begin{cases}
      C(\gamma,\lambda)(1+\|\sigma\|_{L^\infty})^{\frac{2}{\gamma-1}-\lceil
        \lambda\rceil}&\lceil \lambda\rceil\leq\frac{2}{\gamma-1}\\
      C(\gamma,\lambda)\rho_{\min}^{-(\lceil \lambda\rceil-\frac{2}{\gamma-1})}&
      \lceil \lambda\rceil>\frac{2}{\gamma-1}
    \end{cases}\]
  Therefore, we conclude that
  \[\|\rho\|_{C^{\lceil \lambda\rceil}(\text{supp}(\sigma))}\leq
  C(\gamma,\lambda,\|\sigma\|_{L^\infty},\rho_{\min}^{-1}).\]
\end{proof}
\begin{rem}\label{rem:C1}
A similar estimate holds when we replace $L^2$ by $L^\infty$ in
\eqref{composition}.
  \begin{equation}\label{compositioninf}
    \|\mathcal{L}^\lambda(\rho(\sigma)-1)\|_{L^\infty}\leq C(\gamma,\lambda,\|\sigma\|_{L^\infty},\rho_{\min}^{-1}) \|\sigma\|_{C^\lambda}.
  \end{equation}
  We will make use of this estimate for $\lambda\in(0,2)$.
  Note that \eqref{compositioninf} needs to be slightly modified when
  $\lambda=1$, where $\|\sigma\|_{C^1}$ should be replaced by
  $\|\sigma\|_{C^{1+\epsilon}}$ for any $\epsilon>0$ (and the constant
  depends on $\epsilon$). For the sake of simplicity, we will keep the
  compact notation $\|\sigma\|_{C^1}$ throughout the paper.
\end{rem}

\section{Local well-posedness}
The local well-posedness theory of the
system \eqref{eqn1a}-\eqref{eqn3a} can be established by using standard iteration scheme and the compactness argument. For simplicity, in this section we will give only the {\it a priori} energy estimates, and omit the detailed construction of approximation solutions.

We start with the $L^2$ energy estimate.
\begin{lem}[$L^2$ estimate]\label{lem:L2} Let $(\sigma,u)$ be the
  classical solution of $(\ref{eqn1a})-(\ref{eqn2a})$.
  Then the following estimate holds
\begin{align}\label{L2}
\frac{1}{2}\frac{d}{dt}\Big(\|\sigma\|_{L^2}^2+\|u\|_{L^2}^2\Big)+&\beta\|u\|_{L^2}^2+\frac{\rho_{\min}}{2}\|\mathcal{L}^{\alpha}u\|_{L^2}^2\\
  \leq&\,
C(\|\sigma\|_{L^\infty}, \rho_{\min}^{-1})(\|\nabla\cdot u\|_{L^\infty}+\|\sigma\|_{C^{2\alpha}}) (\|u\|_{L^2}^2+\|\nabla \sigma\|_{L^2}^2).\nonumber
\end{align}
\end{lem}

\begin{proof}
Firstly, by multiplying (\ref{eqn1a}) and (\ref{eqn2a}) by $\sigma$ and $u$ respectively, summing up and integrating over $\Omega$, we obtain
\begin{align}\label{gul}
\frac{1}{2}\frac{d}{dt}\int \Big(|\sigma|^2&+|u^2|\Big)\mathrm{d}x+\beta\|u\|_{L^2}^2=-\int \mathcal{L}^{2\alpha}u\cdot u\mathrm{d}x\\
&\quad-\int \Big(\sigma u\cdot \nabla \sigma+u\cdot \nabla u\cdot u\Big)\mathrm{d}x
-\int \Big( \frac{\gamma-1}{2}\sigma+\sqrt{\gamma}\Big)\nabla\cdot(\sigma u)\mathrm{d}x\nonumber\\
&\quad-\int \mathcal{L}^{2\alpha}((\rho(\sigma)-1)u)\cdot u\mathrm{d}x+\int u\mathcal{L}^{2\alpha}(\rho(\sigma)-1)\cdot u\mathrm{d}x=\sum_{i=1}^{5}I_i.\nonumber
\end{align}
We estimate $I_i$ item by item. The commutator estimate \eqref{comm22}
is used for $I_4$, and the composition estimate \eqref{compositioninf}
is used for $I_4$ and $I_5$.
\begin{align*}
I_1&=-\int \mathcal{L}^{2\alpha}u\cdot u\mathrm{d}x=-\int |\mathcal{L}^{\alpha}u|^2\mathrm{d}x;\\
I_2&=-\int \Big(\sigma u\cdot \nabla \sigma+u\cdot \nabla u\cdot u\Big)\mathrm{d}x=-\int\sigma u\cdot \nabla \sigma\mathrm{d}x+\frac{1}{2}\int |u|^2\nabla\cdot u\mathrm{d}x\\
&\leq~ \|\sigma\|_{L^\infty}\|u\|_{L^2}\|\nabla \sigma\|_{L^2}+\frac{1}{2}\|\nabla\cdot u\|_{L^\infty}\|u\|^2_{L^2}\leq C(\|\sigma\|_{L^\infty}+\|\nabla\cdot  u\|_{L^\infty})(\|\nabla \sigma\|_{L^2}^2+\|u\|^2_{L^2});
\end{align*}
\begin{align*}
I_3&=\frac{\gamma-1}{2}\int \sigma u\cdot\nabla\sigma\mathrm{d}x
\leq C\|\sigma\|_{L^\infty}(\|\nabla \sigma\|_{L^2}^2+\|u\|^2_{L^2});\\
I_4&=-\int \mathcal{L}^{2\alpha}((\rho(\sigma)-1)u)\cdot u\mathrm{d}x=-\int \mathcal{L}^{\alpha}((\rho(\sigma)-1)u)\cdot\mathcal{L}^{\alpha} u\mathrm{d}x\\
&=-\int (\rho(\sigma)-1)\mathcal{L}^{\alpha}u\cdot\mathcal{L}^{\alpha} u\mathrm{d}x-\int[\mathcal{L}^{\alpha},\rho(\sigma)-1]u\cdot\mathcal{L}^{\alpha} u\mathrm{d}x\\
&\leq(1-\rho_{\min})\|\mathcal{L}^{\alpha} u\|_{L^2}^{2}+C\|\mathcal{L}^{\alpha}(\rho(\sigma)-1)\|_{L^\infty}\|u\|_{L^2}\|\mathcal{L}^{\alpha} u\|_{L^2}\\
&\leq\left(1-\frac{\rho_{\min}}{2}\right)\|\mathcal{L}^{\alpha} u\|_{L^2}^{2}+C(\|\sigma\|_{L^\infty},\rho_{\min}^{-1})\|\sigma\|_{C^\alpha}\|u\|_{L^2}^2;\\
I_5&=\int u\mathcal{L}^{2\alpha}(\rho(\sigma)-1)\cdot u\mathrm{d}x\leq
     \|\mathcal{L}^{2\alpha}(\rho(\sigma)-1)\|_{L^\infty}\|u\|_{L^2}^2
     \leq C(\|\sigma\|_{L^\infty}, \rho_{\min}^{-1})\|\sigma\|_{C^{2\alpha}}\|u\|_{L^2}^2.
\end{align*}
Collecting the above estimates into (\ref{gul}), we obtain the
$L^2$-estimate of $(\sigma, u)$ \eqref{L2}.
\end{proof}

Next,we provide the $\dot{H}^s$ energy estimate.
\begin{lem}[$\dot{H}^s$ energy estimate]\label{lem:Hs}
Let $(\sigma,u)$ be a classical solution of
$(\ref{eqn1a})-(\ref{eqn2a})$.
Then the following estimate holds for $s>\frac{N}{2}+\max\{1,2\alpha\}$
\begin{align}\label{guhs}
  \frac{1}{2}\frac{d}{dt}\Big(\|\mathcal{L}^s\sigma\|_{L^2}^2+&\|\mathcal{L}^su\|_{L^2}^2\Big)+\beta\|\mathcal{L}^s u\|_{L^2}^2+\frac{\rho_{\min}}{2}\|\mathcal{L}^{s+\alpha}u\|_{L^2}^2\\
&\leq C(\|\sigma\|_{L^\infty},\rho_{\min}^{-1})(\|\nabla
    u\|_{L^\infty}+\|\sigma\|_{C^{\max\{1,2\alpha\}}})(\|u\|_{H^s}^2+\|\mathcal{L}^s\sigma\|_{L^2}^2+\|\mathcal{L}^{s+\alpha-1}\sigma\|_{L^2}^2).\nonumber
\end{align}
\end{lem}

\begin{proof}
We apply $\mathcal{L}^s$ to (\ref{eqn1a}), (\ref{eqn2a}),  multiply the resulting identities by $\mathcal{L}^s \sigma,\ \mathcal{L}^s u$ respectively,  and integrate over $\Omega$ to obtain
\begin{align}\label{guh}
\frac{1}{2}\frac{d}{dt}\int&\Big( |\mathcal{L}^s\sigma|^2+|\mathcal{L}^su|^2\Big)\mathrm{d}x+\beta\|\mathcal{L}^s u\|_{L^2}^2\nonumber\\
&=-\int\mathcal{L}^{s+2\alpha}u\cdot\mathcal{L}^su\mathrm{d}x-\int \Big( \mathcal{L}^s(u\cdot \nabla \sigma)\mathcal{L}^s \sigma+\mathcal{L}^s(u\cdot \nabla u)\cdot \mathcal{L}^s u\Big)\mathrm{d}x\nonumber\\
&\quad -\frac{\gamma-1}{2}\int \Big(\mathcal{L}^s(\sigma\nabla\cdot u) \mathcal{L}^s\sigma+\mathcal{L}^s(\sigma \nabla \sigma)\cdot \mathcal{L}^su\Big)\mathrm{d}x\nonumber\\
&\quad-\int\mathcal{L}^{s+2\alpha}((\rho(\sigma)-1)u)\cdot\mathcal{L}^su\mathrm{d}x+\int\mathcal{L}^s(u\mathcal{L}^{2\alpha}(\rho(\sigma)-1))\cdot \mathcal{L}^su\mathrm{d}x=\sum_{i=1}^{5}J_i.
\end{align}
We directly get
\begin{align}
J_1&=-\int\mathcal{L}^{s+2\alpha}u\cdot\mathcal{L}^su\mathrm{d}x=-\int\Big|\mathcal{L}^{s+\alpha}u\Big|^2\mathrm{d}x.\label{i1}
\end{align}
Applying the commutator estimate \eqref{comm21}, we get
\begin{align}
J_2&=-\int \Big( \mathcal{L}^s(u\cdot \nabla \sigma)\mathcal{L}^s \sigma+\mathcal{L}^s(u\cdot \nabla u)\cdot \mathcal{L}^s u\Big)\mathrm{d}x\nonumber\\
&=-\int \Big( u\cdot \mathcal{L}^s\nabla\sigma\mathcal{L}^s \sigma +[\mathcal{L}^s, u]\cdot\nabla \sigma\mathcal{L}^s \sigma +u\cdot \mathcal{L}^s\nabla u\cdot \mathcal{L}^s u +[\mathcal{L}^s,u]\cdot \nabla u\cdot \mathcal{L}^s u\Big)\mathrm{d}x\nonumber\\
&=\int \Big( \frac{1}{2}\nabla\cdot u |\mathcal{L}^s\sigma|^2 -[\mathcal{L}^s, u]\cdot\nabla \sigma\mathcal{L}^s \sigma +\frac{1}{2}\nabla\cdot u |\mathcal{L}^s u|^2 -[\mathcal{L}^s,u]\cdot \nabla u\cdot \mathcal{L}^s u\Big)\mathrm{d}x\nonumber\\
&\leq\frac12\|\nabla\cdot u\|_{L^\infty} (\|\mathcal{L}^s\sigma\|_{L^2}^2+\|\mathcal{L}^s u\|_{L^2}^2)+ \|[\mathcal{L}^s, u]\cdot\nabla \sigma\|_{L^2}\|\mathcal{L}^s \sigma \|_{L^2} +\|[\mathcal{L}^s,u]\cdot \nabla u\|_{L^2}\| \mathcal{L}^s u\|_{L^2}\nonumber\\
&\leq \frac12\|\nabla\cdot u\|_{L^\infty} (\|\mathcal{L}^s\sigma\|_{L^2}^2+\|\mathcal{L}^s u\|_{L^2}^2)+ C(\|\mathcal{L}^s  u\|_{L^2}\|\nabla \sigma\|_{L^\infty}+\|\nabla u\|_{L^\infty}\|\mathcal{L}^{s-1}\nabla\sigma\|_{L^2})\|\mathcal{L}^s \sigma \|_{L^2}\nonumber\\
&\qquad +C(\|\mathcal{L}^s  u\|_{L^2}\|\nabla u\|_{L^\infty}+\|\nabla u\|_{L^\infty}\|\mathcal{L}^{s-1}\nabla u\|_{L^2})\| \mathcal{L}^s u\|_{L^2}\nonumber\\
&\leq C(\|\nabla u\|_{L^\infty}+\|\nabla\sigma\|_{L^\infty}) (\|\mathcal{L}^s\sigma\|_{L^2}^2+\|\mathcal{L}^s u\|_{L^2}^2).\label{i2}
\end{align}
Similarly we can do the estimates for $J_3$,
\begin{align}
J_3& = -\frac{\gamma-1}{2}\int \Big(\mathcal{L}^s(\sigma\nabla\cdot u) \mathcal{L}^s\sigma+\mathcal{L}^s(\sigma \nabla \sigma)\cdot \mathcal{L}^su\Big)\mathrm{d}x \nonumber\\
& = -\frac{\gamma-1}{2}\int \Big(\sigma \nabla\cdot (\mathcal{L}^su \mathcal{L}^s\sigma)+ [\mathcal{L}^s,\sigma](\nabla\cdot u) \mathcal{L}^s\sigma+[\mathcal{L}^s,\sigma ]\nabla \sigma\cdot \mathcal{L}^su\Big)\mathrm{d}x \nonumber\\
&\leq C\|\nabla\sigma\|_{L^\infty}(\|\mathcal{L}^s\sigma\|_{L^2}^2+\|\mathcal{L}^s u\|_{L^2}^2)+ C(\|\mathcal{L}^s  \sigma\|_{L^2}\|\nabla \cdot u\|_{L^\infty}+\|\nabla \sigma\|_{L^\infty}\|\mathcal{L}^{s-1}\nabla\cdot u\|_{L^2})\|\mathcal{L}^s \sigma \|_{L^2}\nonumber\\
&\qquad +C(\|\mathcal{L}^s  \sigma\|_{L^2}\|\nabla \sigma\|_{L^\infty}+\|\nabla \sigma\|_{L^\infty}\|\mathcal{L}^{s-1}\nabla\sigma\|_{L^2})\| \mathcal{L}^s u\|_{L^2}\nonumber\\
&\leq C(\|\nabla\cdot u\|_{L^\infty}+\|\nabla\sigma\|_{L^\infty}) (\|\mathcal{L}^s\sigma\|_{L^2}^2+\|\mathcal{L}^s u\|_{L^2}^2).\label{i3}
\end{align}
The estimates for $J_4$ and $J_5$ have to be handled together,
applying commutator estimates \eqref{comm21} and \eqref{comm3}, one obtains
\begin{align*}
J_4+J_5=& -\int\mathcal{L}^{s+2\alpha}((\rho(\sigma)-1)u)\cdot\mathcal{L}^su\mathrm{d}x+\int\mathcal{L}^s(u\mathcal{L}^{2\alpha}(\rho(\sigma)-1))\cdot \mathcal{L}^su\mathrm{d}x\\
=& -\int\Big(\mathcal{L}^{s+\alpha}((\rho(\sigma)-1)u)-\mathcal{L}^{s-\alpha}(u\mathcal{L}^{2\alpha}(\rho(\sigma)-1))\Big)\cdot \mathcal{L}^{s+\alpha}u\mathrm{d}x\\
=& -\int(\rho(\sigma)-1)\mathcal{L}^{s+\alpha}u\cdot \mathcal{L}^{s+\alpha}u\mathrm{d}x\\
& +\int\Big([\mathcal{L}^{s+\alpha},u,\rho(\sigma)-1]-[\mathcal{L}^{s-\alpha},u]\big(\mathcal{L}^{2\alpha}(\rho(\sigma)-1)\big)\Big)\cdot \mathcal{L}^{s+\alpha}u\mathrm{d}x\\
\leq&\,\, (1-\rho_{\min})\|\mathcal{L}^{s+\alpha}u\|_{L^2}^2 \\
& +\Big(\|\mathcal{L}^{s+\alpha-1}u\|_{L^2} \|\nabla(\rho(\sigma)-1)\|_{L^\infty}+\|\nabla u\|_{L^\infty}\|\mathcal{L}^{s+\alpha-1}(\rho(\sigma)-1)\|_{L^2}\Big)\|\mathcal{L}^{s+\alpha}u\|_{L^2}\\
& + \Big(\|\mathcal{L}^{s-\alpha}u\|_{L^2} \|\mathcal{L}^{2\alpha}(\rho(\sigma)-1)\|_{L^\infty}
+\|\nabla u\|_{L^\infty}\|\mathcal{L}^{s+\alpha-1}(\rho(\sigma)-1)\|_{L^2}\Big)\|\mathcal{L}^{s+\alpha}u\|_{L^2}.
\end{align*}
We remark that in the third equality, there is a cancelation of the term
\begin{equation}\label{cancelation}
  \int u\mathcal{L}^{s+\alpha}(\rho(\sigma)-1)
  \cdot \mathcal{L}^{s+\alpha}u\mathrm{d}x,
\end{equation}
which itself can not be controlled.
This is the place where the commutator structure \eqref{com} is
crucially used.

Next, we apply the composition estimates \eqref{composition},
\eqref{compositioninf},
and use the facts $s-\alpha<s$, $s+\alpha-1<s$ to get
\begin{align}\label{i45}
  &J_4+J_5\leq\left(1-\frac{\rho_{\min}}{2}\right)\|\mathcal{L}^{s+\alpha}u\|_{L^2}^2\\
  &+C(\|\sigma\|_{L^\infty},\rho_{\min}^{-1})(\|\nabla
    u\|_{L^\infty}+\|\sigma\|_{C^{\max\{1,2\alpha\}}})(\|\mathcal{L}^{s+\alpha-1}u\|_{L^2}^2+\|\mathcal{L}^{s-\alpha}u\|_{L^2}^2+\|\mathcal{L}^{s+\alpha-1}\sigma\|_{L^2}^2)\nonumber\\
  &\leq\left(1-\frac{\rho_{\min}}{2}\right)\|\mathcal{L}^{s+\alpha}u\|_{L^2}^2
        +C(\|\sigma\|_{L^\infty},\rho_{\min}^{-1})(\|\nabla
    u\|_{L^\infty}+\|\sigma\|_{C^{\max\{1,2\alpha\}}})(\|u\|_{H^s}^2+\|\mathcal{L}^{s+\alpha-1}\sigma\|_{L^2}^2)\nonumber
\end{align}
Finally, we plug in the estimates \eqref{i1}-\eqref{i45} to
\eqref{guh} and finish the $\dot{H}^s$ estimate \eqref{guhs}.
\end{proof}

With the energy estimates above, we are ready to prove the local well-posedness theory.
\begin{proof}[Proof of Theorem \ref{local}]
  We combine the estimates \eqref{L2} and \eqref{guhs} for the full
  $H^s$ estimate
  \begin{align}\label{Hsest}
    \frac{1}{2}\frac{d}{dt}\Big(\|\sigma\|_{H^s}^2+\|u\|_{H^s}^2\Big)+&\beta\|u\|_{H^s}^2+\frac{\rho_{\min}}{2}\|u\|_{H^{s+\alpha}}^2\\
  \leq&\, C(\|\sigma\|_{L^\infty},\rho_{\min}^{-1}) (\|\nabla
    u\|_{L^\infty}+\|\sigma\|_{C^{\max\{1,2\alpha\}}})(\|u\|_{H^s}^2+\|\sigma\|_{H^s}^2).\nonumber
  \end{align}
  For $s>\frac{N}{2}+\max\{1,2\alpha\}$, Sobolev embedding implies
  \[\|\nabla u\|_{L^\infty}+\|\sigma\|_{C^{\max\{1,2\alpha\}}}\leq
    C(\|u\|_{H^s}+\|\sigma\|_{H^s})\leq \sqrt{2}C(\|u\|_{H^s}^2+\|\sigma\|_{H^s}^2)^{1/2}.\]
  As $\|\sigma_0\|_{L^\infty}$ and $\rho_{\min}^{-1}(0)$ are bounded,
  there exists a time $T_0$ such that $\|\sigma(t,\cdot)\|_{L^\infty}$
  and $\rho_{\min}^{-1}(t)$ are uniformly bounded for
  $t\in[0,T_0]$. Then, there is a universal constant $C(T_0)$ such that
  \[C(\|\sigma(t,\cdot)\|_{L^\infty},\rho_{\min}^{-1}(t))\leq
    C(T_0).\]
  Consequently, we have
  \[\frac{1}{2}\frac{d}{dt}\Big(\|\sigma\|_{H^s}^2+\|u\|_{H^s}^2\Big)\leq
    C(T_0) \Big(\|\sigma\|_{H^s}^2+\|u\|_{H^s}^2\Big)^{3/2},\quad
    \forall~t\in[0,T_0].\]
  Standard ODE theory implies the existence of time $T\leq T_0$
  such that $\|\sigma\|_{H^s}^2+\|u\|_{H^s}^2$ is bounded for
  $t\in[0,T]$.

  Integrating \eqref{Hsest} in time, we get
  \[\int_0^T\|u(t,\cdot)\|_{H^{s+\alpha}}^2\mathrm{d}t\leq
    \max_{t\in[0,T]}\rho_{\min}^{-1}(t)
  \left(\|\sigma_0\|_{H^s}^2+\|u_0\|_{H^s}^2+2C(T_0)\int_0^T \left(\|\sigma\|_{H^s}^2+\|u\|_{H^s}^2\right)\mathrm{d}t\right)<+\infty,\]
which implies $u\in L^2([0,T], (H^{s+\alpha})^N)$.

Finally, we show that the regularity  \eqref{localreg} holds as long
as condition \eqref{BKM} is satisfied.
Applying the Gronwall's inequality on \eqref{Hsest},  we get
 \[\|\sigma(t,\cdot)\|_{H^s}^2+\|u(t,\cdot)\|_{H^s}^2\leq
   \Big(\|\sigma_0\|_{H^s}^2+\|u_0\|_{H^s}^2\Big)
 \exp\left[\int_0^t C(\|\sigma\|_{L^\infty},\rho_{\min}^{-1}) (\|\nabla
   u\|_{L^\infty}+\|\sigma\|_{C^{\max\{1,2\alpha\}}})\mathrm{d}t\right].\]
Therefore, the solution exists up to time $T$ as long as
\begin{equation}\label{BKM2}
  \int_0^T C(\|\sigma\|_{L^\infty},\rho_{\min}^{-1}) (\|\nabla
  u\|_{L^\infty}+\|\sigma\|_{C^{\max\{1,2\alpha\}}})\mathrm{d}t<+\infty.
\end{equation}
From \eqref{eqn1o}, we get
$(\partial_t+u\cdot\nabla)\rho=-(\nabla\cdot u)\rho$. Therefore, we
can bound $\rho(t,x)$ by
\[\rho_{\min}(0)\exp\left[-\int_0^t\|\nabla\cdot u(\tau,\cdot)\|_{L^\infty}\mathrm{d}\tau\right]
  \leq \rho(t,x) \leq
 \|\rho_0\|_{L^\infty}\exp\left[\int_0^t\|\nabla\cdot u(\tau,\cdot)\|_{L^\infty}\mathrm{d}\tau\right].\]
Hence, if \eqref{BKM} holds, $\rho_{\min}^{-1}(t)$ is bounded for
$t\in[0,T]$.
Also, using the relation \eqref{sigma}, we have that
$\|\sigma(t,\cdot)\|_{L^\infty}$ is bounded for $t\in[0,T]$.
Hence, $C(\|\sigma\|_{L^\infty},\rho_{\min}^{-1})$ is bounded by a
universal constant, and condition \eqref{BKM2} is reduced to
\eqref{BKM}. This finishes the proof.
\end{proof}

\section{Global regularity for small data}
In this section, we discuss the global well-posedness of the system
\eqref{eqn1a}-\eqref{eqn3a}, with small initial data
\[\|\sigma_0\|_{H^s}^2+\|u_0\|_{H^s}^2\leq\delta_0^2,\]
for a small parameter $\delta_0>0$ to be chosen.

We will show that the solution stays small. In particular, the
following lemma suffices to show Theorem \ref{global}.
\begin{lem}[Propagation of smallness]\label{lem:small}
  Assume $(\sigma, u)$ is a classical solution of
  \eqref{eqn1a}-\eqref{eqn3a}, such that
  \begin{equation}\label{xiao}
    \underset{{0\leq t\leq T}}{\sup}\Big(\|\sigma(t,\cdot)\|_{H^s}^2+\|u(t,\cdot)\|_{H^s}^2 \Big)\leq \delta^2,
  \end{equation}
  where $\delta>0$ is sufficiently small.
  Then the solution stays small, namely, there exists a universal
   constant $C_0>1$, such that
   \begin{equation}\label{decay}
     \|\sigma(t,\cdot)\|_{H^s}^2+\|u(t,\cdot)\|_{H^s}^2\leq
     C_0^2(\|\sigma_0\|_{H^s}^2+\|u_0\|_{H^s}^2),\quad\forall~t\in[0,T].
   \end{equation}
\end{lem}

\begin{proof}[Proof of Theorem \ref{global}]
  Let $\delta$ be the small parameter in Lemma \ref{lem:small}.
  Let $T_{\max}$ be the maximum time that the solution stays small,
  namely
  \[T_{\max}:=\inf\big\{t\geq0~:~\|\sigma(t,\cdot)\|_{H^s}^2+\|u(t,\cdot)\|_{H^s}^2>\delta^2\big\}.\]
  Suppose $T_{\max}<\infty$, then by continuity argument,
  \[\|\sigma(T_{\max},\cdot)\|_{H^s}^2+\|u(T_{\max},\cdot)\|_{H^s}^2=\delta^2..\]
  However, if we pick $\delta_0<\frac{\delta}{C_0}$, from Lemma
  \ref{lem:small} we get
  \[\|\sigma(T_{\max},\cdot)\|_{H^s}^2+\|u(T_{\max},\cdot)\|_{H^s}^2\leq C_0^2(
    \|\sigma_0\|_{H^s}^2+\|u_0\|_{H^s}^2)\leq
    C_0^2\delta_0^2<\delta^2.\]
  This leads to a contradiction.
  Therefore, $T_{\max}=+\infty$, finishing the proof.
\end{proof}

We are left to show Lemma \ref{lem:small}.
Recall the $H^s$ estimate (combination of Lemmas \ref{lem:L2} and
\ref{lem:Hs})
\begin{equation}\label{energyest}
  \frac{d}{dt}\Big(\|\sigma\|_{H^s}^2+\|u\|_{H^s}^2\Big)+\beta\|u\|_{H^s}^2
  +\frac{1}{4}\|\mathcal{L}^\alpha u\|_{H^s}^2
\leq  C\delta(\|u\|_{H^s}^2+\|\nabla\sigma\|_{H^{s-1}}^2).
\end{equation}
It is slightly different from \eqref{Hsest}, as we will comment in the following.

First, the smallness condition \eqref{xiao} and Sobolev embedding
implies
\[\|\nabla u\|_{L^\infty}+\|\sigma\|_{C^{\max\{1,2\alpha\}}}<
  C\delta.\]
Second, picking $\delta$ small enough, we have $\rho_{\min}>\frac{1}{2}$
and $\|\sigma\|_{L^\infty}<C\delta$. Then, $C(\|\sigma\|_{L^\infty},
\rho_{\min}^{-1})$ has a uniform bound. Therefore, we can drop its
dependence on $\sigma$.
Finally, as we assume $s>2-\alpha$ so that
$s+\alpha-1\in(1,s)$, the term
$\|\mathcal{L}^{s+\alpha-1}\sigma\|_{L^2}$ in \eqref{guhs} can be
controlled by $\|\nabla\sigma\|_{H^{s-1}}$.
It does not depend on $\|\sigma\|_{L^2}$.
This is important as we do not have dissipation estimate on
$\|\sigma\|_{L^2}$.

To show \eqref{decay}, we need to control the right hand side of
\eqref{energyest}.

The term $\|u\|_{H^s}^2$ can be controlled by the
damping or the dissipation. If $\beta>0$, then we can pick
$\delta<\frac{\beta}{2C}$, such that
\[C\delta\|u\|_{H^s}^2\leq \frac{\beta}{2}\|u\|_{H^s}^2.\]
If $\beta=0$, we can use the dissipation term to control $\|\nabla
u\|_{H^{s-1}}^2$. The $\|u\|_{L^2}^2$ can only be controlled in the
case $\Omega=\mathbb{T}^N$ by Poincare inequality
$\|u\|_{L^2}\leq C\|\nabla u\|_{L^2}$, which implies
$\|u\|_{H^s}\leq C\|\mathcal{L}^\alpha u\|_{H^s}$. Similarly, we can
pick $\delta$ small enough so that
\[C\delta\|u\|_{H^s}^2\leq \frac{1}{8}\|\mathcal{L}^\alpha u\|_{H^s}^2.\]

To sum up, there exists a positive number $\nu>0$, such that
\begin{equation}\label{guh3}
  \frac{d}{dt}\Big(\|\sigma\|_{H^s}^2+\|u\|_{H^s}^2\Big)+\nu\|u\|_{H^{s+\alpha}}^2
  \leq C\delta \|\nabla\sigma\|_{H^{s-1}}^2.
\end{equation}

To control the remaining term $\|\nabla\sigma\|_{H^{s-1}}^2$, we adopt
the idea introduced in \cite{STW}, using cross terms to obtain
dissipation estimates for $\sigma$.

\begin{lem}[Dissipation estimates for $\sigma$]\label{lem:cross}
\begin{equation}\label{hh1}
	\frac{d}{dt}\int
        u\cdot\nabla\sigma\mathrm{d}x+\frac{3\sqrt{\gamma}}{4}\|\nabla
        \sigma\|_{L^2}^2\leq  C \|u\|_{H^s}^2,
	\end{equation}
\begin{equation}\label{hhs}
\frac{d}{dt}\int \mathcal{L}^{s-1}u\cdot\nabla \mathcal{L}^{s-1}\sigma\mathrm{d}x+\frac{\sqrt{\gamma}}{2}\|\mathcal{L}^s \sigma\|_{L^2}^2\leq C\|u\|_{H^{s+\alpha}}^2 +C\delta\|\nabla\sigma\|_{L^2}^2.
\end{equation}
\end{lem}

\begin{proof} First, we can directly calculate to obtain
  \[
    \frac{d}{dt}\int u\cdot\nabla\sigma\mathrm{d}x=-\int \sigma_t \nabla\cdot u\mathrm{d}x+\int u_t \cdot\nabla\sigma\mathrm{d}x.
  \]
  The two terms separately are
\begin{align*}
  -\int \sigma_t \nabla\cdot u\mathrm{d}x=&
  \int \left(u\cdot\nabla\sigma +
    \left(\frac{\gamma-1}{2}\sigma+\sqrt{\gamma}\right)\nabla\cdot u\right)(\nabla\cdot u)\mathrm{d}x\\
  \leq&\, \|u\|_{L^\infty}\|\nabla\sigma\|_{L^2}\|\nabla\cdot u\|_{L^2}
  +\left(\frac{\gamma-1}{2}\|\sigma\|_{L^\infty}+\sqrt{\gamma}\right)\|\nabla\cdot u\|_{L^2}^2\\
  \leq &\, C\|\nabla\cdot u\|_{L^2}^2 + C\delta \|\nabla\sigma\|_{L^2}^2
\end{align*}
and
\begin{align*}
\int u_t\cdot\nabla\sigma\mathrm{d}x=&-\sqrt{\gamma}\|\nabla \sigma\|_{L^2}^2-\beta\int u\cdot\nabla\sigma\mathrm{d}x\\
&-\int u\cdot\nabla u\cdot\nabla\sigma\mathrm{d}x-\frac{\gamma-1}{2}\int \sigma\nabla \sigma\cdot\nabla\sigma\mathrm{d}x-\int \mathcal{L}^{2\alpha}u\cdot\nabla\sigma\mathrm{d}x\\
&-\int \mathcal{L}^{2\alpha}((\rho(\sigma)-1)u)\cdot\nabla\sigma\mathrm{d}x+\int u\mathcal{L}^{2\alpha}(\rho(\sigma)-1)\cdot\nabla\sigma\mathrm{d}x\\
  \leq & -\frac{7\sqrt{\gamma}}{8}\|\nabla \sigma\|_{L^2}^2 +
         C(\|u\|_{L^2}^2+\|u\cdot\nabla u\|_{L^2}^2+\|\mathcal{L}^{2\alpha} u\|_{L^2}^2)\\
 & + C(\|\mathcal{L}^{2\alpha}((\rho(\sigma)-1)u)\|_{L^2}^2+
 \|u\mathcal{L}^{2\alpha}(\rho(\sigma)-1)\|_{L^2}^2)\\
 \leq& -\frac{7\sqrt{\gamma}}{8}\|\nabla \sigma\|_{L^2}^2 +C\|u\|_{H^{2\alpha}}^2\\
       &+C\Big(\|\mathcal{L}^{2\alpha} u\|_{L^2}\|\rho(\sigma)-1\|_{L^\infty}+
       \|u\|_{L^\infty}\|\mathcal{L}^{2\alpha}(\rho(\sigma)-1)\|_{L^2}\Big)^2\\
       \leq & -\frac{7\sqrt{\gamma}}{8}\|\nabla \sigma\|_{L^2}^2 + C\|u\|_{H^s}^2.
\end{align*}
The fractional Leibniz rule \eqref{Leibnitz} is used in the second
last inequality, followed by the estimates
\[\|\mathcal{L}^{2\alpha}(\rho(\sigma)-1)\|_{L^2}\leq
C\|\sigma\|_{\dot H^{2\alpha}}<\delta,\quad \|u\|_{L^\infty}\leq C\|u\|_{H^s}.\]
Sum up the above estimates, and choose $\delta$ small enough so
that $C\delta\|\nabla\sigma\|_{L^2}^2$ in the first term is absorbed by
$\frac{\sqrt\gamma}{8}\|\nabla\sigma\|_{L^2}^2$ in the second term.
We end up with \eqref{hh1}.

Now we prove \eqref{hhs}, by a direct computation we obtain
\[
  \frac{d}{dt}\int\mathcal{L}^{s-1}u\cdot\nabla\mathcal{L}^{s-1}\sigma\mathrm{d}x
  =-\int\mathcal{L}^{s-1}\sigma_t\mathcal{L}^{s-1}\nabla\cdot u\mathrm{d}x
  +\int \mathcal{L}^{s-1}u_t \cdot\nabla\mathcal{L}^{s-1}\sigma\mathrm{d}x.
\]
For the first term, apply the Leibniz rule \eqref{Leibnitz} and get
\begin{align*}
  -\int\mathcal{L}^{s-1}\sigma_t\mathcal{L}^{s-1}\nabla\cdot u\mathrm{d}x=&\,
  \sqrt{\gamma}\int|\mathcal{L}^{s-1}\nabla\cdot u|^2\mathrm{d}x+\int\mathcal{L}^{s-1}(u\cdot\nabla\sigma)\mathcal{L}^{s-1}\nabla\cdot u\mathrm{d}x\\
  &+\frac{\gamma-1}{2}\int\mathcal{L}^{s-1}(\sigma\nabla\cdot
    u)\mathcal{L}^{s-1}\nabla\cdot u\mathrm{d}x\\
  \leq&\, C\|\mathcal{L}^su\|_{L^2}^2+C \|\mathcal{L}^su\|_{L^2}
        \Big(\|\mathcal{L}^{s-1}u\|_{L^2}\|\nabla\sigma\|_{L^\infty}+
        \|\mathcal{L}^{s-1}\nabla\sigma\|_{L^2}\|u\|_{L^\infty}\Big)\\
  & +C \|\mathcal{L}^su\|_{L^2}
        \Big(\|\mathcal{L}^{s-1}\nabla\cdot u\|_{L^2}\|\sigma\|_{L^\infty}+
    \|\mathcal{L}^{s-1}\sigma\|_{L^2}\|\nabla\cdot u\|_{L^\infty}\Big)\\
 \leq&\, C\|\mathcal{L}^{s} u\|_{L^2}^2 +C\delta ( \|\mathcal{L}^{s-1}u\|_{L^2}^2+\|\mathcal{L}^{s}\sigma\|_{L^2}^2+\|\mathcal{L}^{s-1}\sigma\|_{L^2}^2).\\
  \leq&\, C\|u\|_{H^s}^2+C\delta(\|\mathcal{L}^{s}\sigma\|_{L^2}^2+\|\nabla\sigma\|_{L^2}^2).
\end{align*}
For the second term,
\begin{align*}
\int \mathcal{L}^{s-1}u_t\cdot&\,\nabla\mathcal{L}^{s-1}\sigma\mathrm{d}x=
-\sqrt{\gamma}\|\nabla\mathcal{L}^{s-1} \sigma\|_{L^2}^2-\beta\int \mathcal{L}^{s-1}u\cdot\nabla\mathcal{L}^{s-1}\sigma\mathrm{d}x-\int \mathcal{L}^{s-1+2\alpha}u\cdot\nabla\mathcal{L}^{s-1}\sigma\mathrm{d}x\\
&-\int \mathcal{L}^{s-1}(u\cdot\nabla u)\cdot\nabla\mathcal{L}^{s-1}\sigma\mathrm{d}x-\frac{\gamma-1}{2}\int \mathcal{L}^{s-1}(\sigma\nabla \sigma)\cdot\nabla\mathcal{L}^{s-1}\sigma\mathrm{d}x\\
&-\int \mathcal{L}^{s-1+2\alpha}((\rho(\sigma)-1)u)\cdot\nabla\mathcal{L}^{s-1}\sigma\mathrm{d}x+\int \mathcal{L}^{s-1}(u\mathcal{L}^{2\alpha}(\rho(\sigma)-1))\cdot\nabla\mathcal{L}^{s-1}\sigma\mathrm{d}x\\
\leq &\, -\frac{3\sqrt{\gamma}}{4}\|\mathcal{L}^s \sigma\|_{L^2}^2 + C\|\mathcal{L}^{s-1}u\|_{L^2}^2+ C\|\mathcal{L}^{s-1+2\alpha} u\|_{L^2}^2\\
&+C( \|\mathcal{L}^{s-1}(u\cdot \nabla u)\|_{L^2}^2 +\|\mathcal{L}^{s-1}(\sigma\nabla\sigma)\|_{L^2}^2)\\
& + C\left\|\mathcal{L}^{s-1+2\alpha}((\rho(\sigma)-1)u)-\mathcal{L}^{s-1}(u\mathcal{L}^{2\alpha}(\rho(\sigma)-1))\right\|_{L^2}^2 =:\sum^3_{i=1}K_i.
\end{align*}
Since $s-1+2\alpha<s+\alpha$, $K_1$ can be simply estimated by
\[K_1\leq -\frac{3\sqrt{\gamma}}{4}\|\mathcal{L}^{s} \sigma\|_{L^2}^2+C\|u\|_{H^{s+\alpha}}^2,\]
We estimate $K_2$ by using the fractional Leibniz rule \eqref{Leibnitz},
\begin{align*}
K_2	\leq &\,  C\Big( \|\mathcal{L}^{s-1}u\|_{L^2}^2 \|\nabla u\|_{L^\infty}^2+\|\mathcal{L}^{s-1}\nabla u\|_{L^2}^2 \|u\|_{L^\infty}^2+\|\mathcal{L}^{s-1}\sigma\|_{L^2}^2\|\nabla \sigma\|_{L^\infty}^2+ \|\mathcal{L}^{s-1}\nabla\sigma\|_{L^2}^2\| \sigma\|_{L^\infty}^2\Big)\\
\leq &\,  C\delta (\|u\|_{H^s}^2+ \|\mathcal{L}^{s} \sigma\|_{L^2}^2 +\|\nabla\sigma\|_{L^2}^2 ).
\end{align*}
For $K_3$, we use a similar cancelation as \eqref{cancelation}, and
apply commutator estimates
\begin{align*}
  K_3\leq &\,
    C\| [\mathcal{L}^{s-1+2\alpha}, u](\rho(\sigma)-1) \|_{L^2}^2
    +C\|[\mathcal{L}^{s-1}, u]\mathcal{L}^{2\alpha}(\rho(\sigma)-1)\|_{L^2}^2\\
 \leq &\,
   C\left(\|\mathcal{L}^{s-1+2\alpha}u\|_{L^2}\|\rho(\sigma)-1\|_{L^\infty}
   +\|\nabla u\|_{L^\infty}\|\mathcal{L}^{s-2+2\alpha}(\rho(\sigma)-1)\|_{L^2}\right)^2\\
 &+C\left(\|\mathcal{L}^{s-1}u\|_{L^2}\|\mathcal{L}^{2\alpha}(\rho(\sigma)-1)\|_{L^\infty}
    +\|\nabla u\|_{L^\infty}\|\mathcal{L}^{s-2+2\alpha}(\rho(\sigma)-1)\|_{L^2}\right)^2\\
  \leq&\,C\delta\|u\|_{H^s}^2+C\|\nabla u\|_{L^\infty}^2\|\mathcal{L}^{s-2+2\alpha}\sigma\|_{L^2}^2
   \leq C\delta\|u\|_{H^s}^2.
\end{align*}
Note that as $s>\frac{N}{2}+\max\{1,2\alpha\}$ and $s>2-\alpha$, one
can easily check $0<s-2+2\alpha<s$.
In the last inequality, we have used
$\|\mathcal{L}^{s-2+2\alpha}\sigma\|_{L^2}<\|\sigma\|_{H^s}<\delta$,
and $\|\nabla u\|_{L^\infty}\leq C\|u\|_{H^s}$.

Collect all the estimates of $K_i$, $i=1,2,3$, and choose $\delta$
small so that the $\|\mathcal{L}^s\sigma\|_{L^2}$ term in $K_2$ be absorbed
in to $K_1$, We end up with the desired estimate \eqref{hhs}.
\end{proof}

Adding \eqref{hh1} and \eqref{hhs} and pick $\delta$ small enough, we
get
\begin{equation}\label{hhall}
	\frac{d}{dt}\int
        \left(u\cdot\nabla\sigma+\mathcal{L}^{s-1}u\cdot\nabla \mathcal{L}^{s-1}\sigma\right)\mathrm{d}x+\frac{\sqrt{\gamma}}{2}\|\nabla
        \sigma\|_{H^{s-1}}^2\leq  C \|u\|_{H^s}^2,
\end{equation}
which can be used to control the term $\|\nabla\sigma\|_{H^{s-1}}$ in
\eqref{guh3}. Indeed, multiplying \eqref{hhall} by
$\dfrac{4C\delta}{\sqrt{\gamma}}$, adding it to \eqref{guh3}, and
choosing $\delta$ small, we obtain
\begin{equation}\label{Yt}
  \frac{d}{dt}Y(t)
+\frac{\nu}{2}\|u\|_{H^{s+\alpha}}^2+C\delta\|\nabla
\sigma\|_{H^{s-1}}^2\leq 0,
\end{equation}
where
\[Y(t):=\|\sigma(t,\cdot)\|_{H^s}^2+\|u(t,\cdot)\|_{H^s}^2+\dfrac{4C\delta}{\sqrt{\gamma}}\int(u\cdot \nabla\sigma+\mathcal{L}^{s-1}u\cdot \nabla \mathcal{L}^{s-1}\sigma)\mathrm{d}x.\]
Note that if $\delta$ is small enough, $Y(t)\geq0$ and is equivalent
to  $\|\sigma(t,\cdot)\|_{H^s}^2+\|u(t,\cdot)\|_{H^s}^2$. Indeed,
there exists a constant $C_0>1$ such that
\begin{equation}\label{unh3}
  C_0^{-1}\Big(\|\sigma(t,\cdot)\|_{H^s}^2+\|u(t,\cdot)\|_{H^s}^2\Big)\leq
  Y(t)\leq C_0\Big(\|\sigma(t,\cdot)\|_{H^s}^2+\|u(t,\cdot)\|_{H^s}^2\Big).
\end{equation}

Integrating \eqref{Yt} directly in time, we have
\begin{equation}\label{Yest}
  Y(t)+\int_0^t \left(\frac{\nu}{2}\|u(\tau,\cdot\|_{H^{s+\alpha}}^2
    +C\delta\|\nabla\sigma(\tau,\cdot)\|_{H^{s-1}}^2\right)
  \mathrm{d}\tau\leq Y(0).
\end{equation}
With the help  of (\ref{unh3}), we obtain that for any $t\in[0,T]$,
\[\|\sigma(t,\cdot)\|_{H^s}^2+\|u(t,\cdot)\|_{H^s}^2\leq C_0Y(t)\leq
  C_0Y(0)\leq C_0^2\Big(\|\sigma_0\|_{H^s}^2+\|u_0\|_{H^s}^2\Big).\]
This ends the proof of Lemma \ref{lem:small}.

\begin{rem}\label{rem:longtime}
  Another outcome of \eqref{Yest} is that
  \[\int_0^\infty
    \Big(\|u(t,\cdot)\|_{H^{s+\alpha}}^2+\|\nabla\sigma(t,\cdot)\|_{H^{s-1}}^2\Big)
    \mathrm{d}t\leq Y(0)<+\infty.\]
  Therefore, $\|u\|_{H^s}$ and $\|\nabla\sigma\|_{H^{s-1}}$ decay
  to zero as $t\to\infty$.
  In particular, $\|u\|_{L^\infty}\to0$ implies the flocking
  phenomenon (to be more precise, velocity alignment).
  Note that our estimates does not imply any
  decay for $\|\sigma\|_{L^2}$. The convergence rate of $\|u\|_{H^s}$
  and $\|\nabla\sigma\|_{H^{s-1}}$ will also depend on the long time
  behavior of $\|\sigma\|_{L^2}$.
  More discussions and results for the case
  $\Omega=\mathbb{T}^N$ will be offered in the next section.
\end{rem}

\section{Long time behavior}

In this part, we study the large-time behavior of global classical
solutions $(\rho, u)$ to the system \eqref{eqn1o}-\eqref{eqn3} for the
case $\Omega=\mathbb{T}^N$. Without lose of generality we assume
$|\mathbb{T}^N|=1$.

Firstly, in order to obtain the dissipation estimate of $\rho$, we
denote the function
\begin{eqnarray*}
h(\rho)=\int^\rho_1 \dfrac{z^\gamma-1}{z^2}\mathrm{d}z=\left\{\begin{array}{ll} \rho \ln\rho-\rho+1, &\gamma=1,\\[3mm] \dfrac{1}{\gamma-1}\rho^\gamma-\dfrac{\gamma}{\gamma-1}\rho+1,&\gamma>1,
\end{array}\right.
\end{eqnarray*}
so that $h(1)=h'(1)=0$ and $p'(\rho)=h''(\rho)\rho$.
One can easily check that when $\rho$ is close to 1,
$h(\rho)\sim(\rho-1)^2$. Indeed, we state the following lemma.

\begin{lem}[\cite{fang}]
  Let $\delta>0$ be a small parameter, and $|\rho-1|<\delta$.  Then,
  there exists constants $c_1, c_2$, depending only on $\delta$, such that
  \begin{equation}\label{equiv}
    c_1(\rho-1)^2\leq h(\rho)\leq c_2(\rho-1)^2.
  \end{equation}
\end{lem}

Now, we are ready to prove Theorem \ref{large}.

\noindent\emph{Step 1: Dissipation of the physical energy.}
 Let $(\rho,u)$ be a global classical solutions to
 \eqref{eqn1o}-\eqref{eqn2o}, we will establish the decay estimate of
 the physical energy:
\begin{equation}\label{es}
\frac{d}{dt}\int \Big(\frac{1}{2}\rho |u|^2+h(\rho)\Big)\mathrm{d}x+\beta\int\rho|u|^2\mathrm{d}x+\frac{1}{2}\iint\phi(x-y)(u(x)-u(y))^2\rho(x)\rho(y)\mathrm{d}x\mathrm{d}y\leq 0.
\end{equation}
Using the symmetry of $\phi(\cdot)$ and equation \eqref{eqn1o}-\eqref{eqn2o}, it is obtained that
\begin{align}\label{ru}
&\frac{d}{dt}\int\rho |u|^2\mathrm{d}x=\int\rho_t|u|^2+2\int\rho u \cdot u_t\mathrm{d}x \nonumber\\
&=-\int{\rm div}(\rho u)|u|^2\mathrm{d}x-2\int\rho u\Big(u\cdot\nabla u +\frac{\nabla p(\rho)}{\rho}+\beta u+\int\phi(x-y)(u(x)-u(y))\rho(y)\mathrm{d}y\Big)\nonumber\\
&=-2\int u\cdot\nabla p(\rho)\mathrm{d}x-2\beta\int\rho|u|^2\mathrm{d}x-2\iint\rho(x)u(x)\phi(x-y)(u(x)-u(y))\rho(y)\mathrm{d}x\mathrm{d}y\nonumber\\
&=-2\int u\cdot\nabla p(\rho)\mathrm{d}x-2\beta\int\rho|u|^2\mathrm{d}x-\iint\phi(x-y)(u(x)-u(y))^2\rho(x)\rho(y)\mathrm{d}x\mathrm{d}y.
\end{align}
Furthermore, applying the definition of $h(\rho)$ , we can compute that
\begin{equation}\label{hr}
\frac{d}{dt}\int h(\rho)\mathrm{d}x
=\int h^\prime(\rho)\rho_t\mathrm{d}x=-\int h^\prime(\rho){\rm div}(\rho u)\mathrm{d}x=\int h^{\prime\prime}(\rho)\rho\nabla\rho\cdot u\mathrm{d}x=\int u\cdot\nabla p(\rho)\mathrm{d}x.
\end{equation}
Combining \eqref{ru}-\eqref{hr}, we obtain \eqref{es}.

\noindent\emph{Step 2: Dissipation of $\|\rho-1\|_{L^2}$.}
We first define a stream function $\psi$ which solves the Poisson
equation
\[
  -\Delta\psi=\rho-1,\qquad \int_{\mathbb{T}^N}\psi\mathrm{d}x=0.
\]
It is uniquely defined on $\mathbb{T}^N$ since $\rho-1$ has zero
mean (condition \eqref{rhobar}).

Similar as Lemma \ref{lem:cross}, we introduce a small cross-term to
the physical energy, and define 
\[
  V_{\varepsilon}=\int \Big(\frac{1}{2}\rho|u|^2+h(\rho)+\varepsilon\rho u\cdot \nabla\psi\Big)\mathrm{d}x.
\]
Note that the cross term can be controled by
\begin{align}\label{5.10}
\left|\int\rho u \cdot\nabla\psi\mathrm{d}x\right|&\leq \|\rho\|_{L^\infty}^{\frac{1}{2}}\|\sqrt{\rho}u\|_{L^2}\|\nabla\psi\|_{L^2}\leq \frac{1}{2}\int\rho|u|^2\mathrm{d}x+\frac{C\|\rho\|_{L^\infty}}{2}\int(\rho-1)^2\mathrm{d}x,
\end{align}
where Poincar\'e inequality is used so that $\|\nabla\psi\|_{L^2}\leq
C\|\nabla^{\otimes2}\psi\|_{L^2}\leq C\|\rho-1\|_{L^2}$. Then, 
it follows from \eqref{equiv}, \eqref{5.10} and for a sufficient small $\varepsilon$
that there exists a constant $C_1>0$ which depends on $\varepsilon$ and $\|\rho\|_{L^{\infty}}$, such that
\begin{equation}\label{vv}
\frac{1}{C_1}\Big(\int \rho|u|^2\mathrm{d}x+\int|\rho-1|^2\mathrm{d}x\Big)\leq V_\varepsilon\leq C_1\Big(\int \rho|u|^2\mathrm{d}x+\int(\rho-1)^2\mathrm{d}x\Big),
\end{equation}
namely, $V_\varepsilon$ is equivalent to the physical energy.

Applying \eqref{es}, It is easy to check that
\begin{align}\label{vw}
\frac{d}{dt}V_\varepsilon+W_\varepsilon\leq0,
\end{align}
where
\begin{align*}
W_{\varepsilon}&=\beta\int\rho|u|^2\mathrm{d}x+\frac{1}{2}\iint\phi(x-y)(u(x)-u(y))^2\rho(x)\rho(y)\mathrm{d}x\mathrm{d}y
+\varepsilon\int\nabla\psi\cdot\nabla p(\rho)\mathrm{d}x\\
&+\varepsilon\int\nabla\psi\cdot\nabla(\rho u\otimes u)\mathrm{d}x-\varepsilon\int\rho u \cdot \nabla\psi_t \mathrm{d}x\\
&+\varepsilon\iint\nabla\psi(x)\cdot\phi(x-y)(u(x)-u(y))\rho(x)\rho(y)\mathrm{d}x\mathrm{d}y=\sum_{i=1}^{6}L_i.
\end{align*}

We investigate further for the terms $L_1$ to $L_6$ in $W_\varepsilon$.
Notice that $L_1$ and $L_2$ are positive. The positivity of $L_3$ can be obtained in the following:
\[
  L_3=-\varepsilon\int\Delta\psi \big(p(\rho)-p(1)\big)\mathrm{d}x=\varepsilon\int(\rho-1)(p(\rho)-p(1))\mathrm{d}x\geq \varepsilon\int (\rho-1)^2\mathrm{d}x,
\]
where we have used the fact that $\frac{\rho^\gamma-1}{\rho-1}\geq1$ for
any $\rho\geq0$ and $\gamma\geq1$.
This term produces dissipation for $\|\rho-1\|_{L^2}$.

$L_4$ can be controlled by the dissipation as follows
\begin{align*}
|L_4|=&~\varepsilon\left|\int\nabla\otimes\nabla\psi : (\rho u\otimes
       u)\mathrm{d}x\right|\leq\varepsilon\|\rho-1\|_{L^2}\|\rho u\otimes
       u\|_{L^2}\\ \leq
  &~\frac{\varepsilon}{4}\|\rho-1\|_{L^2}^2+\varepsilon\|\rho u\otimes u\|_{L^2}^2\leq\frac{\varepsilon}{4}\|\rho-1\|_{L^2}^2+\varepsilon\|\rho\|_{L^{\infty}}\|u\|^2_{L^{\infty}}\int\rho |u|^2\mathrm{d}x.
\end{align*}

The term $L_5$ can be estimated in the following
\[
|L_5|=\Big|\varepsilon\int\rho u \cdot \nabla(-\Delta)^{-1}\nabla\cdot(\rho u) \mathrm{d}x\Big|\leq \varepsilon \|\rho u\|_{L^2}\|\nabla(-\Delta)^{-1}\nabla\cdot(\rho u)\|_{L^2}\leq \varepsilon \|\rho u\|_{L^2}^2\leq \varepsilon\|\rho\|_{L^\infty}\int\rho |u|^2\mathrm{d}x.
\]

Finally, for $L_6$ we have
\begin{align*}
|L_6|&=\frac{\varepsilon}{2}\left|\iint\big(\nabla\psi(x)-\nabla\psi(y)\big)\phi(x-y)\cdot(u(x)-u(y))\rho(x)\rho(y)\mathrm{d}x\mathrm{d}y\right|\\
     &\leq\theta\varepsilon\iint\frac{|\nabla\psi(x)-\nabla\psi(y)|^2}{|x-y|^{n+2\alpha}}\mathrm{d}x\mathrm{d}y+\frac{C}{\theta}\varepsilon\|\rho\|^2_{L^\infty}\iint\phi(x-y)\rho(x)\rho(y)(u(x)-u(y))^2\mathrm{d}x\mathrm{d}y,
\end{align*}
where the first part
\[\iint\frac{|\nabla\psi(x)-\nabla\psi(y)|^2}{|x-y|^{n+2\alpha}}\mathrm{d}x\mathrm{d}y=
  \|\nabla\psi\|_{\dot{H}^\alpha}^2\leq C\|\psi\|_{\dot{H}^2}^2\leq
  C\|\rho-1\|_{L^2}^2,\quad
\forall~\alpha\in(0,1).\]
Then, choosing $\theta$ appropriately, we obtain that
\[
  |L_6|\leq \frac{\varepsilon}{4}\|\rho-1\|_{L^2}^2
+C\varepsilon\|\rho\|^2_{L^\infty}
\iint\phi(x-y)\rho(x)\rho(y)(u(x)-u(y))^2\mathrm{d}x\mathrm{d}y.
\]

Collecting all the estimates of $L_i$, and picking $\varepsilon$ small
enough (which depends on $\|\rho\|_{L^\infty}$ and
$\|u\|_{L^\infty}$), we get a lower bound on $W_\varepsilon$
\begin{align}\label{Wlow}
W_\varepsilon\geq&~
\Big(\beta-C\varepsilon\|\rho\|_{L^{\infty}}\|u\|^2_{L^{\infty}}-\varepsilon\|\rho\|_{L^{\infty}}\Big)\int\rho |u|^2\mathrm{d}x\nonumber\\
&~+\Big(\frac{1}{2}-C\varepsilon\|\rho\|^2_{L^\infty}\Big)\iint\phi(x-y)\rho(x)\rho(y)(u(x)-u(y))^2\mathrm{d}x\mathrm{d}y
+\frac{\varepsilon}{2}\int (\rho-1)^2\mathrm{d}x.\nonumber\\
  \geq&~\frac{\beta}{2}\int\rho |u|^2\mathrm{d}x
        +\frac{1}{4}\iint\phi(x-y)\rho(x)\rho(y)(u(x)-u(y))^2\mathrm{d}x\mathrm{d}y
        +\frac{\varepsilon}{2}\int (\rho-1)^2\mathrm{d}x.
\end{align}

When there is no damping, namely $\beta=0$, we can use the second term
in \eqref{Wlow} to produce dissipation for the kinetic energy.

\begin{lem}
  There exists a constant $\tilde{C}>0$, depending on $\alpha$, such that
  \begin{equation}\label{TNdis}
\frac{1}{4}\iint\phi(x-y)\rho(x)\rho(y)(u(x)-u(y))^2\mathrm{d}x\mathrm{d}y\geq \tilde{C}\int\rho|u|^2\mathrm{d}x.
\end{equation}
\end{lem}
\begin{proof}
  Write the integral on the left hand side of \eqref{TNdis} in
  $\mathbb{T}^{2N}$
  \[\frac{1}{4}\iint_{\mathbb{T}^{2N}}\phi(x-y)\rho(x)\rho(y)(u(x)-u(y))^2\mathrm{d}x\mathrm{d}y
  =\frac{1}{4}\iint_{\mathbb{T}^{2N}}\phi_P(x-y)\rho(x)\rho(y)(u(x)-u(y))^2\mathrm{d}x\mathrm{d}y\]
with periodic kernel $\phi_P$
\[\phi_P(x)=\sum_{k\in\mathbb{Z}^{N}}\phi(x+k),\quad\forall~x\in\mathbb{T}^N.\]
Clearly, $\phi_P$ has a positive lower bound
(e.g. $\phi_P(x)>\displaystyle\min_{\|x\|_\infty<\frac12}\phi(x)>0$). Let
us denote the bound by $\phi_m$. It only depends on $\alpha$ (and
$\mathbb{T}^N$ as well, but we have set $|\mathbb{T}^N|=1$ here). Then, we have
\begin{align*}
  &\frac{1}{4}\iint_{\mathbb{T}^{2N}}\phi_P(x-y)\rho(x)\rho(y)(u(x)-u(y))^2\mathrm{d}x\mathrm{d}y
\geq
    \frac{\phi_m}{4}\iint_{\mathbb{T}^{2N}}\rho(x)\rho(y)(u(x)-u(y))^2\mathrm{d}x\mathrm{d}y\\
  &\quad=\frac{\phi_m}{2}\left(\int_{\mathbb{T}^N}\rho(x)\mathrm{d}x
    \int_{\mathbb{T}^N}\rho(x)|u(x)|^2\mathrm{d}x-\left|
    \int_{\mathbb{T}^N}\rho(x)u(x)\mathrm{d}x\right|^2\right)=\frac{\phi_m}{2} \int_{\mathbb{T}^N}\rho|u|^2\mathrm{d}x,
\end{align*}
where we have used \eqref{rhobar} and \eqref{zeromom}. This implies \eqref{TNdis} with $\tilde{C}=\frac{\phi_m}{2}$.
\end{proof}

Thus, we deduce that
\begin{equation}\label{Wdelta}
W_\varepsilon\geq 
\Big(\frac\beta2+\tilde{C}\Big)\int\rho |u|^2\mathrm{d}x+\frac{\varepsilon}{2}\int (\rho-1)^2\mathrm{d}x.
\end{equation}

\noindent\emph{Step 3: The decay estimate.}
Combining the inequalities in \eqref{vv}, \eqref{vw} and
\eqref{Wdelta}, we deduce that there exists a constant $\mu>0$, which
depend on $\gamma$ and $\varepsilon$, such that
\[
  \frac{d}{dt}V_\varepsilon+\mu V_\varepsilon\leq0.
\]
Applying Gronwall's inequality and \eqref{vv}, we get the desired
exponential decay
\begin{align*}
\int\rho|u|^2\mathrm{d}x+\int(\rho -1)^2\mathrm{d}x&\leq C_1V_\varepsilon(t)\leq C_1V_\varepsilon(0)e^{-\mu t}\nonumber\\
&\leq C_1^2\left(\int\rho_0|u|_0^2\mathrm{d}x+\int(\rho_0-1)^2\mathrm{d}x \right) e^{-\mu t}.
\end{align*}
This finishes the proof of \eqref{expDecay},

\noindent\emph{Step 4: Decay on higher order norms.}
To obtain exponential decay for $\|\sigma\|_{H^s}$ and $\|u\|_{H^s}$,
we start with the control on $\|\sigma\|_{L^2}$.
From the relation \eqref{sigma}, we have
\[\|\sigma\|_{L^2}^2=\int(\sigma(\rho)-\sigma(1))^2\mathrm{d}x\leq
  \|\sigma'\|_{L^\infty(\text{supp}(\rho))}^2\int(\rho-1)^2\mathrm{d}x\leq
  C\|\rho-1\|_{L^2}^2,\]
where $\sigma'(\rho)=\sqrt{\gamma}\rho^{\frac{\gamma-3}{2}}$ is
bounded provided that $\rho$ is bounded and $\rho_{\min}>0$. This is
guaranteed by Theorem \ref{global} with a small enough
$\delta_0$. Therefore, we have
\[\|\sigma(t,\cdot)\|_{L^2}^{2}\leq Ce^{-\mu t}.\]

Next, recall the estimate \eqref{Yt} and add
$C\delta\|\sigma\|_{L^2}^2$ on both sides of the inequality
\[
  \frac{d}{dt}Y(t)+\frac{\nu}{2}\|u\|_{H^{s+\alpha}}^2+C\delta\|\sigma\|_{H^s}^2\leq C\delta\|\sigma\|_{L^2}^2,
\]
Apply \eqref{unh3} and get
\[\frac{d}{dt}Y(t)+\mu_1Y(t)\leq C\delta\|\sigma\|_{L^2}^2\leq
  C\delta e^{-\mu t},
  \]
where we can choose $\mu_1=C_0^{-1}\min\{\frac{\nu}{2}, C\delta\}$.
This implies
\[Y(t)\leq \big(Y(0)+C\delta\big)e^{-\min\{\mu_1,\mu\} t}.\]
Using \eqref{unh3} again, we conclude
\[\|\sigma(t,\cdot)\|_{H^s}^2+\|u(t,\cdot)\|_{H^s}^2\leq C_0Y(t)\leq C_0 \big(Y(0)+C\delta\big)e^{-\min\{\mu_1,\mu\} t}\leq  C\delta_0e^{-\min\{\mu_1,\mu\} t}.\]
This ends the proof of \eqref{expDecayh}.


\end{document}